\documentclass[a4paper,10pt]{amsart}

\usepackage[utf8]{inputenc}
\usepackage[OT2,T1]{fontenc}
\usepackage[english]{babel}

\usepackage{amsmath}
\usepackage{amsfonts}
\usepackage{amssymb}
\usepackage{amsthm}
\usepackage{mathrsfs}
\usepackage{listings}
\usepackage[mathcal]{euscript}

\usepackage{multirow}

\usepackage[dvipsnames]{xcolor}
\usepackage{graphicx}
\usepackage[all]{xy}
\usepackage{hyperref}

\usepackage{colortbl}
\definecolor{mygray}{gray}{0.92}

\usepackage{array}
\newcolumntype{C}[1]{>{\centering\arraybackslash$}p{#1}<{$}}

\usepackage{breqn}
\newcounter{myequation}[equation]

\theoremstyle{plain}
\newtheorem{theorem}{Theorem}[section]

\newtheorem{proposition}[theorem]{Proposition}
\newtheorem{lemma}[theorem]{Lemma}
\newtheorem{corollary}[theorem]{Corollary}

\theoremstyle{definition}
\newtheorem{definition}[theorem]{Definition}

\theoremstyle{remark}
\newtheorem{remark}[theorem]{Remark}
\newtheorem{example}[theorem]{Example}

\numberwithin{equation}{section}

\def\epsilon{\varepsilon}
\def\theta{\vartheta}
\def\tilde{\widetilde}

\def\Magma{\textsc{Magma}}

\def\P{\mathbb{P}\,}
\def\Q{\mathbb{Q}}

\def\Z{\mathbb{Z}}

\usepackage{bm}

\author[E. Lorenzo Garc\'ia]{Elisa Lorenzo Garc\'ia}
\address{fLaboratoire IRMAR, Universit\'e de Rennes 1\\
	Campus de Beaulieu, b\^at. 22-23,
	35042 Rennes C\'edex,
	France
} \email{elisa.lorenzogarcia@univ-rennes1.fr}

\address{Institut de Math\'ematiques, Universit\'e de Neuch\^atel, Rue Emile-Argand 11, 2000, Neuch\^atel, Switzerland
} \email{elisa.lorenzo@unine.ch}



\begin{document}

\title{Construction of Brauer-Severi varieties}

\begin{abstract}
In this paper we give an algorithm for computing equations of Brauer-Severi varieties over fields of characteristic $0$. As an example we show the equations of all Brauer-Severi surfaces defined over $\mathbb{Q}$.
\end{abstract}

\maketitle

\section{Introduction}

The first who systematically studied Brauer–Severi varieties was
Châtelet in his seminal work \cite{Cha} and under the name of "variétés de Brauer". The term "Severi–Brauer
variety" comes from Beniamino Segre \cite{Segre}, who suggested that Châtelet had omitted previous work by Severi  \cite{Severi}, where he  studied Brauer-Severi varieties in a more classical
geometric context.

In the literature one finds different theoretical constructions
of Brauer-Severi varieties: for the classical approach of Châtelet via
varieties of left ideals embedded into Grassmannians, which gives a canonical construction, see \cite{Cha}, \cite{Jacobson} or \cite{Knus}. When trying to produce explicit equations of Brauer-Severi varieties, this projective embedding is far from being "optimal": for instance, Brauer-Severi varieties of dimension 1 are realized not as
plane conics, but as curves in $\mathbb{P}^5$ defined by $31$ equations, 
see \cite[p. 113]{Jacobson}. 

Another approach is that of Grothendieck, which is
based on general techniques in descent theory. It does not give explicit information on the projective
embedding. But when trying to compute it yields the same one that in Châtelet idea: see \cite{Jahnel}.

Even if not canonical, since it is going to depend on the representant of the cocycle class we choose, we will follow the Twisting Theory approach.

\begin{definition}(Brauer-Severi variety) Let $K$ be a perfect field and $X/K$ be a projective irreducible smooth variety of dimension $n$, we say that $X$ is a Brauer-Severi variety if there exist an isomorphism $X_K\simeq_{\overline{K}}\mathbb{P}^{n}_{K}$. Let us denote the set of Brauer-severi varieties of dimension $n$ defined over $K$ up to isomorphism by $\operatorname{BS}_{K}^{n}$.
\end{definition}

Clearly, we have $$\operatorname{BS}_{K}^{n}=\operatorname{Twist}_K(\mathbb{P}^{n}_{K})\simeq\operatorname{H}^1(\operatorname{G}_K,\operatorname{Aut}(\mathbb{P}_{K}^{n}))\simeq\operatorname{H}^1(\operatorname{G}_K,\operatorname{PGL}_{n+1}(\overline{K})),$$ where $\operatorname{G}_K$ denotes the absolute Galois group $\operatorname{Gal}(\overline{K}/K)$. If $K$ is a finite field or the function field of an algebraic curve over an algebraically closed field, then $\operatorname{H}^1(\operatorname{G}_K,\operatorname{PGL}_{n+1}(\overline{K}))$ is trivial (Tsen's Theorem) and there are not non-trivial Brauer-Severi varieties.

The first and only previously known equations of a non-trivial Brauer-Severi variety ($n\geq 2$) were shown in \cite{BaBaEl1}. It is defined over $\Q(\zeta_3)$ where $\zeta_3$ is a third primitive root of unity. 

As application of the explicit construction of a non-trivial Brauer-Severi surface $\mathcal{B}$, see \cite{Pa}, one could construct a nontrivial cubic surface with a Galois stable set of 6 pairwise skew lines starting from $\mathcal{B}$: following Manin \cite[Ch. 4, Sec. 31]{Manin}. And this would help to determine whether the condition $(a)$ in the following Theorem of Swinnerton-Dyer is really needed.

\begin{theorem}(Swinnerton-Dyer, \cite{SD})
	Let S be a smooth cubic surface defined over a number field $K$. S is
	birationally trivial if and only if 
	\begin{itemize}
		\item [(a)] S contains a point defined over K and,
		\item [(b)] S contains a $\operatorname{Gal}(Q/K)$-stable set of 2, 3 or 6 pairwise skew lines.
	\end{itemize}
\end{theorem}

As it was already notice by Swinnerton-Dyer, a smooth cubic surface containing a stable set of $2$ lines contains a rational point, and then it is birationally equivalent to the projective plane. Whether this is also true for a stable set of $3$ or $6$ lines is still unknown. 

In order to compute explicit equations of Brauer-Severi varieties, we need to compute explicit equations of twists of the projective space $\mathbb{P}^n$. As it is shown in \cite{BaBaEl1} for smooth plane curves, in \cite{LomLor} for hyperelliptic curves, and in \cite{Loth} and \cite{Lo} for non-hyperelliptic curves, the best idea to compute equations of twists of a variety $X/K$ is to embed its automorphism group $\operatorname{Aut}(X)$ into $\operatorname{GL}_N(\overline{K})$ for some $N\in\mathbb{N}$ as a $\operatorname{G}_K$-module, and then apply Hilbert's Theorem 90.

Hence, in order to compute non-trivial Brauer-Severi varieties, we will follow this strategy:

\begin{itemize}
	\item We will describe the set $\operatorname{Twist}_K(\mathbb{P}^{n}_{K})\simeq\operatorname{H}^1(\operatorname{G}_K,\operatorname{PGL}_{n+1}(\overline{K}))$.
	\item We will give an embedding of $\operatorname{G}_K$-modules $\operatorname{PGL}_{n+1}(\overline{K})\subseteq\operatorname{GL}_N(\overline{K})$ for some $N\in\mathbb{N}$ that will allow us to compute explicit equations for Brauer-Severi varieties by explicitly using Hilbert's Theorem 90.
\end{itemize}  

\subsection*{Acknowledgements} The author thanks Francesc Bars for suggesting to her the problem of finding equations for Brauer-Severi varieties. She also thanks Jeoren Sijsling for carefully reading a preliminary version of the manuscript and providing lots of useful comments and remarks. Finally, she would like to thank René Pannekoek for his ideas concerning Section $6$ and for pointing her out Theorem \ref{thm-singularmodel} for the case $n=2$.

\section{Brauer-Severi varieties and central simple algebras}

 The set of isomorphism classes of central simple algebras of dimension $n^2$ over $K$ and split over $L$ is denoted by $\operatorname{Az}_{n}^{L/K}$. The set of isomorphism classes of central simple algebras of dimension $n^2$ over $K$ is denoted  by $\operatorname{Az}_{n}^{K}$. 

\begin{theorem}[Serre, chap. X, \S5, Prop. 8, \cite{SeL}]\label{thm-cocycle}
Let $L/K$ be a finite Galois extension of fields, $G=\operatorname{Gal}(L/K)$ its Galois group, and $n$ be a natural number. Then there is a natural bijection of pointed sets
$$
\begin{array}{cccc}
a^{L/K}_n: & \operatorname{Az}_{n}^{L/K} & \xrightarrow{\simeq} & \operatorname{H}^1(G,\operatorname{PGL}_{n}(L)).
\end{array}
$$
\end{theorem}

Notice that previous Theorem implies $$\operatorname{Az}_{n}^{K}=\cup_{L/K}\operatorname{Az}_{n}^{L/K}\simeq\operatorname{H}^1(\operatorname{G}_K,\operatorname{PGL}_n(\overline{K}))\simeq\operatorname{BS}^n_K.$$

It is well-known that for $n=2,3$ all the algebras in $\operatorname{Az}_{n}^{K}$ are cyclic algebras \cite{Wed}. We show the equivalent definition for cyclic simple central algebras given in \cite{GilSza}. 

\begin{proposition}\label{prop-cyclic}There is a bijection between the set of isomorphism classes of cyclic algebras of degree $n$ over $K$  and the set of equivalent classes of pairs $(\chi,a)$ where $\chi:\operatorname{Gal}(L/K)\simeq\mathbb{Z}/n\mathbb{Z}$ is a group isomorphism with $L$ a cyclic Galois extension of degree $n$ of $K$ and $a\in K^{*}$. The equivalent relation is $(\chi,a)\sim(\chi',a')$ if and only if $\chi=\chi'$ and $a'a^{-1}\in \operatorname{Nm}_{L/K}(L^{*})$. Given a pair $(\chi,a)$, the corresponding algebra is given (by Theorem \ref{thm-cocycle}) by the cocycle in $\operatorname{H}^1(\operatorname{Gal}(L/K),\operatorname{PGL}_n(L))$ that maps
\small{
$$
\chi^{-1} (1)\mapsto\begin{pmatrix}
0 & 1 & 0 &... & 0\\
0 & 0 & 1 &... & 0\\
  &   &   &... &  \\
0 & 0 & 0 &... & 1\\
\alpha & 0 & 0 &... & 0 
\end{pmatrix}.
$$}
\end{proposition}

\section{The key embedding}
The following lemma will be the key point for constructing equations defining Brauer-Severi varieties via Hilbert's Theorem $90$. The $n$-Veronese embedding,  $V_n:\,\mathbb{P}^n\hookrightarrow\mathbb{P}^{m}$ with $m=\binom{2n+1}{n}-1$ induces an  embedding $\operatorname{PGL}_{n+1}(\overline{K})\hookrightarrow\operatorname{PGL}_{m+1}(\overline{K})$ of $\operatorname{Gal}(\overline{K}/K)$-modules, see for instance Theorem $5.2.2$ in \cite{GilSza} for a slightly variation of it. We go a little bit further:

\begin{lemma}\label{lemma_key} There exist an embedding of $\operatorname{Gal}(\overline{K}/K)$-modules $\iota_n:\,\operatorname{PGL}_{n+1}(\overline{K})\hookrightarrow\operatorname{GL}_{m+1}(\overline{K})$ where $m=\binom{2n+1}{n}-1$.
\end{lemma}

\begin{remark} The case $n=1$ is Proposition $3.5$ in \cite{LomLor} and it was used to compute twists of hyperelliptic curves. In this case, the Veronese embedding has degree $n+1=2$ and it is given by
	$$
	V_1:\,\mathbb{P}^1\rightarrow\mathbb{P}^2:\,(x:y)\mapsto(x^2:xy:y^2).
	$$
It induces the embedding	
	$$
	\iota_1:\,\operatorname{PGL}_{2}(\overline{K})\hookrightarrow\operatorname{GL}_3(\overline{K}):\,\left[\begin{pmatrix}\alpha & \beta \\ \gamma & \delta\end{pmatrix}\right]\mapsto\frac{1}{\operatorname{det}(A)}\begin{pmatrix}\alpha^2 & 2\alpha\beta & \beta^2 \\ \alpha\gamma & \alpha\delta+\beta\gamma & \beta\delta \\ \gamma^2 & 2\gamma\delta & \delta^2\end{pmatrix}.
	$$
	 The case $n=2$ is implicitly used in \cite{BaBaEl1} for computing the only previous known equations for a non-trivial Brauer-Severi variety of dimension greater than $1$.
\end{remark}

\begin{proof} (Lemma \ref{lemma_key}) Let us consider the Veronese embedding of dimension $n$ and degree $n+1$: $V_n:\,\mathbb{P}^n\hookrightarrow\mathbb{P}^m$ with $m=\binom{2n+1}{n}-1$. We name the coordinates as follows $V_n:\,\mathbb{P}^n\rightarrow\mathbb{P}^m:\,(x_0:...:x_n)\mapsto(\omega_0:...:\omega_m)$, where the $\omega_k$ are equal to the products $\omega_{x_{0}^{\alpha_0}\dots x_{n}^{\alpha_n}}=\prod_{i}x_{i}^{\alpha_i}$ with $\sum_{i}\alpha_i=n+1$ in alphabetical order. 
	
The embedding $V_n$ induces another embedding on automorphism groups $[\iota_n]:\,\operatorname{PGL}_{n+1}(\overline{K})\hookrightarrow\operatorname{PGL}_{m+1}(\overline{k})$. We will see that indeed we can lift it to $$\iota_n:\,\operatorname{PGL}_{n+1}(\overline{K})\hookrightarrow\operatorname{GL}_{m+1}(\overline{K}).$$ Let be $[A]=[a_{ij}]\in\operatorname{PGL}_{n+1}(\overline{K})$, then $[\iota_n]([A])=[(L^k)_{k=0...m}]$ is the matrix whose lines are $L^k$, where again named in alphabetical order, the coordinates of $L^k$ are given by the formula
$$
\prod_i(\sum_j a_{ij}x_j)^{\alpha_i}=\sum_{(\beta_0...\beta_n)}L^k_{\beta_0...\beta_n}x_{0}^{\beta_0}...x_{n}^{\beta_n}.
$$

Hence, the matrix $[\iota_n]([A])$ is a matrix whose entries are polynomials of degree $n$ in the entries of $A$. We can now fix a lift of $[\iota_n]([A])$ to $\operatorname{GL}_m(\overline{K})$ by doing $\iota_n([A])=\frac{1}{\operatorname{det}({A})}(L^k)_{k=0...m}$. This is an embedding of $\operatorname{Gal}(\overline{K}/K)$-modules.
\end{proof}

\begin{remark} The anticanonical sheaf in $\mathbb{P}^n$ is equal to $\mathcal{O}(n+1)$ \cite[8.20.1]{Hart} and it gives the Veronese embedding of degree $n+1$ of $\mathbb{P}^n$ into $\mathbb{P}^m$ with $m=\binom{2n+1}{n}-1$. So, previous embedding can be seen as the natural action of the automorphism group of $\mathbb{P}^n$ on the vector space of global section of the anticanonical sheaf $\mathcal{O}(n+1)$.
\end{remark}

\begin{proposition}\label{prop-eqP} The equations of the image of the Veronese embedding $V_n:\,\mathbb{P}^n\rightarrow\mathbb{P}^m:\,(x_0:...:x_n)\mapsto(\omega_0:...:\omega_m)$, where the $\omega_k$ are equal to the products $\omega_{x_{0}^{\alpha_0}\dots x_{n}^{\alpha_n}}=\prod_{i}x_{i}^{\alpha_i}$ with $\sum_{i}\alpha_i=n+1$ in alphabetical order, are
	$$
	\omega_{x_{0}^{\alpha_0}\dots x_{n}^{\alpha_n}}\omega_{x_{n}^{n+1}}^{n-\alpha_n}=\omega_{x_{0}x_{n}^n}^{\alpha_0}\dots\omega_{x_{n-1}x_n^{n}}^{\alpha_{n-1}},
	$$
	together with the equations given by permuting the indices by $\{\sigma, \sigma^1,...\sigma^{n}\}$ where $\sigma:\,\{0,1,...,n\}\mapsto\{1,2,...,n,0\}$.
\end{proposition}

\begin{proof} Let us call $V$ the variety in $\mathbb{P}^m$ defined by this set of equations that we call $\mathcal{F}$. If $\omega_m=\omega_{x_{n}^{n+1}}\neq 0$, we make $\omega_m=x_n=1$, then $x_i=\omega_{x_ix_{n}^{n}}$ and $\omega_{x_{0}^{\alpha_0}\dots x_{n}^{\alpha_n}}=\omega_{x_{0}x_{n}^n}^{\alpha_0}\dots\omega_{x_{n-1}x_n^{n}}^{\alpha_n}$. The map $V_n:\,\mathbb{P}^n\setminus \{x_n=0\}\rightarrow V\setminus\{\omega_m=0\}$ is clearly a bijection. If $\omega_m=0$, at least for one $i$ we have $\omega_{x_{i}^{n+1}}\neq0$ and we repeat the previous argument with the transformed equations. This proves that $V_n(\mathbb{P}^n)=V$ and that $V_n$ is a bijection. We finally check that $V$ is a non-singular variety: For that we need to check that the matrix
$
(\partial f/\partial x_i)_{f\in\mathcal{F},0\leq i\leq n}
$, where $\mathcal{F}$ is the set of equations in the statement of the proposition,	
has rank at least $m-n$. If $\omega_m=\omega_{x_{n}^{n+1}}\neq 0$, the set of $m-n$ rows $\partial/\partial x_i (\omega_{x_{0}^{\alpha_0}\dots x_{n}^{\alpha_n}}\omega_{x_{n}^{n+1}}^{n-\alpha_n}-\omega_{x_{0}x_{n}^n}^{\alpha_0}\dots\omega_{x_{n-1}x_{n}^{n}}^{\alpha_n})_{i=0,..,n}$ with $\alpha_n< n$ has maximal rank $m-n$. If $\omega_m=\omega_{x_{n}^{n+1}}= 0$, we take an $i$ with $\omega_{x_{i}^{n+1}}\neq0$ and we repeat the previous argument with the permutated equations.
\end{proof}

\begin{corollary}\label{cor_eqP2} For $n=1$, we get the equation of the conic $\omega_0\omega_2=\omega^{2}_{1}$. 
	
	For $n=2$, we get the equations:
	$$\begin{array}{cccc}
	\omega_0\omega_{9}^2=\omega_{5}^{3},\,\,\, & \omega_0\omega_{6}^2=\omega_{3}^{3},\,\,\, &\omega_3\omega_{0}^2=\omega_0\omega_{1}^{2}&\\
	\omega_1\omega_{9}^2=\omega_{5}^{2}\omega_8,\,\,\, & \omega_1\omega_{6}^2=\omega_{3}^{2}\omega_6,\,\,\, &\omega_4\omega_{0}^2=\omega_{0}\omega_1\omega_2&\\
	\omega_2\omega_{9}^2=\omega_{5}^{2}\omega_9,\,\,\, & \omega_2\omega_{6}^2=\omega_{3}^{2}\omega_{7},\,\,\, &  \omega_5\omega_{0}^2=\omega\omega_{2}^{2}&\\
	\omega_3\omega_{9}^2=\omega_{5}\omega_8^{2},\,\,\, & \omega_4\omega_{6}^2=\omega_{3}\omega_6\omega_{7},\,\,\, &\omega_6\omega_{0}^2=\omega_1^{3}& \subseteq \mathbb{P}^9\\
	\omega_4\omega_{9}^2=\omega_{5}\omega_8\omega_9,\,\,\, & \omega_5\omega_{6}^2=\omega_{3}\omega_{7}^2,\,\,\, &\omega_7\omega_{0}^2=\omega_{1}^2\omega_2&\\
	\omega_6\omega_{9}^2=\omega_8^{3},\,\,\, & \omega_8\omega_{6}^2=\omega_6\omega_7^{2},\,\,\, &  \omega_8\omega_{0}^2=\omega_1\omega_2^{2}&\\
	\omega_7\omega_{9}^2=\omega_{8}^{2}\omega_9,\,\,\, & \omega_9\omega_{6}^2=\omega_{7}^{3},\,\,\, &  \omega_9\omega_{0}^2=\omega_{2}^{3}&\\
	\end{array}$$
\end{corollary}

\section{The algorithm}\label{Sec-Algorithm}

	Given a cocycle $\bar{\xi}$ in $\operatorname{H}^1(G_k,\operatorname{PGL}_{n+1}(\overline{K}))$, it defines a Brauer-Severi variety as in Theorem \ref{thm-cocycle}. This algorithm gives equations defining the Brauer-Severi variety.
	\begin{enumerate}
	\item Transform the cocycle $\bar{\xi}$ into a cocycle $\xi$ in $\operatorname{H}^1(G_k,\operatorname{GL}_{m+1}(\overline{K}))$ with Lemma \ref{lemma_key}.
	\item Use a explicit version of Hilbert's Theorem $90$ to get  $\phi\in\operatorname{GL}_{m+1}(\overline{K})$ such that $\xi_\sigma=\phi\circ^{\sigma}\phi^{-1}$ for all $\sigma\in\operatorname{G}_K$. This can be done by taking a sufficiently general matrix in $\operatorname{GL}_{m+1}(\overline{K})$ and applying the recipe in \cite[Prop. 3, p. 159]{SeL}, or by searching for a basis of fixed vectors of a special action of $\operatorname{GL}_{m+1}(\overline{K})$ on $\overline{K}^{m+1}$, see \cite[Section 1.2]{Loth} or \cite[Section 3]{Lo}.
	\item Get the equations of the Brauer-Severi variety in $\mathbb{P}^m$ by using the computed $\phi$ in (ii) and the equations of $V_n(\mathbb{P}^{n})\subseteq\mathbb{P}^{m}$ in Proposition \ref{prop-eqP}. 
\end{enumerate}

\section{The case $n=2$ made explicit}

We write down the computations for the case $n=2$. The Veronese embedding of dimension $2$ and degree $3$ looks like
$$
V_2:\,\mathbb{P}^2\rightarrow\mathbb{P}^{9}:\,(x:y:z)\mapsto(x^3:x^2y:x^2z:xy^2:xyz:xz^2:y^3:y^2z:yz^2:z^3).
$$

This embedding induces another one

$$
\iota_2:\,\operatorname{PGL}_{3}(\overline{K})\hookrightarrow\operatorname{GL}_{10}(\overline{K}):\,[A]=\left[\begin{pmatrix}a & b & c \\ d & e & f \\ g & h & i \end{pmatrix}\right]\mapsto\frac{1}{\operatorname{det}(A)}\cdot M(A)
$$

where the matrix $M(A)$ is given by

\vspace{3mm}

\tiny
\begin{changemargin}{-3cm}{-0.5cm} 
$
\begin{pmatrix}
a^3 & 3a^2b & 3a^2c & 3ab^2 & 6abc & 3ac^2 & b^3 & 3b^2c & 3bc^2 & c^3 \\
a^2d & a^2e + 2abd & a^2f + 2acd & 2abe + b^2d &  2abf + 2ace + 2bcd & 2acf + c^2d & b^2e & b^2f + 2bce & 2bcf + c^2e & c^2f \\
a^2g & a^2h + 2abg & a^2i + 2acg & 2abh + b^2g &  2abi + 2ach + 2bcg &
2aci + c^2g & b^2h & b^2i + 2bch & 2bci + c^2h & c^2i \\
ad^2 & 2ade + bd^2 & 2adf + cd^2 & ae^2 + 2bde & 2aef + 2bdf + 2cde &
af^2 + 2cdf & be^2 & 2bef + ce^2 & bf^2 + 2cef & cf^2\\
adg & adh + aeg + bdg & adi + afg +cdg & aeh + bdh +beg & aei +
afh + bdi +bfg + cdh +ceg & afi +cdi + cfg & beh & bei +bfh + ceh & bfi + cei + cfh & cfi\\
ag^2 & 2agh + bg^2 & 2agi + cg^2 & ah^2 + 2bgh &  2ahi + 2bgi + 2cgh &
ai^2 + 2cgi & bh^2 & 2bhi + ch^2 & bi^2 + 2chi & ci^2\\
d^3 & 3d^2e & 3d^2f & 3de^2 &  6def & 3df^2 & e^3 & 3e^2f & 3ef^2 & f^3\\
d^2g & d^2h + 2deg & d^2i + 2dfg & 2deh + e^2g & 2dei + 2dfh + 2efg &
2dfi + f^2g & e^2h & e^2i + 2efh & 2efi + f^2h & f^2i\\
dg^2 & 2dgh + eg^2 & 2dgi + fg^2 & dh^2 + 2egh & 2dhi + 2egi + 2fgh & di^2 + 2fgi & eh^2 & 2ehi + fh^2 & ei^2 + 2fhi & fi^2\\
g^3 & 3g^2h & 3g^2i &  3gh^2 & 6ghi & 3gi^2 & h^3 & 3h^2i & 3hi^2 & i^3
\end{pmatrix}
$
\end{changemargin}

\normalsize

\vspace{0.5cm}

This matrix has been computed with \Magma \cite{magma}, we show below the code that is easily generalizable to other $n\in\mathbb{N}$.

\vspace{0.5cm}

\texttt{K<x,y,z,a,b,c,d,e,f,g,h,i>:=PolynomialRing(Rationals(),12);}

\texttt{X:=ax+by+cz; Y:=dx+ey+fz; Z:=gx+hy+iz;}

\texttt{X^3; X^2Y; X^2Z; XY^2; XYZ; XZ^2; Y^3; Y^2Z; YZ^2; Z^3;}

\vspace{0.5cm}

By Theorem \ref{thm-cocycle}~, the fact that all the elements in $\operatorname{Az}_3^K$ are cyclic and Proposition \ref{prop-cyclic}~, we need to study the image by $\iota_2$ of the matrix
$$
A_{\alpha}=\begin{pmatrix}
0 & 1 & 0 \\ 0 & 0 & 1 \\ \alpha & 0 & 0
\end{pmatrix}.
$$

So, in this case: $b=1$, $f=1$, $g=\alpha$, $a=c=d=e=f=i=0$, and
\small
	$$
	\iota_2(A_{\alpha})=\begin{pmatrix}
	0 & 0 & 0 & 0 & 0 & 0 & 1/\alpha & 0 & 0 & 0 \\
	0 & 0 & 0 & 0 & 0 & 0 & 0 & 1/\alpha & 0 & 0 \\
	0 & 0 & 0 & 1 & 0 & 0 & 0 & 0 & 0 & 0 \\
	0 & 0 & 0 & 0 & 0 & 0 & 0 & 0 & 1/\alpha & 0 \\
	0 & 0 & 0 & 0 & 1 & 0 & 0 & 0 & 0 & 0 \\
	0 & \alpha & 0 & 0 & 0 & 0 & 0 & 0 & 0 & 0 \\
	0 & 0 & 0 & 0 & 0 & 0 & 0 & 0 & 0 & 1/\alpha \\
	0 & 0 & 0 & 0 & 0 & 1 & 0 & 0 & 0 & 0 \\
	0 & 0 & \alpha & 0 & 0 & 0 & 0 & 0 & 0 & 0 \\
	\alpha^2 & 0 & 0 & 0 & 0 & 0 & 0 & 0 & 0 & 0 \\
	\end{pmatrix}
	$$
\normalsize

\begin{lemma}\label{lemma_iso}Let $L/K$ be a degree $3$ Galois extension of number fields. Write $L=K(l_1,l_2,l_3)$ and $\operatorname{Gal}(L/K)=\langle\sigma\rangle$ with $\sigma(l_1)=l_2$ and $\sigma(l_2)=l_3$. Let $\alpha\in K$ and let us define the cocycle $\xi\in\operatorname{H}^1(\operatorname{Gal}(L/K),\operatorname{GL}_3(L))$ by $\xi_\sigma=\iota_2(A_{\alpha})$. Then $\xi_\tau=\phi\circ^{\tau}\phi^{-1}$ for all $\tau\in\operatorname{Gal}(L/K)$ with
	$$
	\phi=\begin{pmatrix}
	l_1 & 0 & 0 & 0 & 0 & 0 & l_2 & 0 & 0 & l_3 \\
	0 & l_1 & 0 & 0 & 0 & l_2 & 0 & l_3 & 0 & 0 \\
	0 & 0 & l_1 & l_2 & 0 & 0 & 0 & 0 & l_3 & 0 \\
	0 & 0 & l_3 & l_1 & 0 & 0 & 0 & 0 & l_2 & 0 \\
	0 & 0 & 0 & 0 & 1 & 0 & 0 & 0 & 0 & 0 \\
	0 & l_2\alpha & 0 & 0 & 0 & l_3\alpha & 0 & l_1\alpha & 0 & 0 \\
	l_3\alpha & 0 & 0 & 0 & 0 & 0 & l_1\alpha & 0 & 0 & l_2\alpha \\
	0 & l_3\alpha & 0 & 0 & 0 & l_1\alpha & 0 & l_2\alpha & 0 & 0 \\
	0 & 0 & l_2\alpha & l_3\alpha & 0 & 0 & 0 & 0 & l_1\alpha & 0 \\
	l_2\alpha^2 & 0 & 0 & 0 & 0 & 0 & l_3\alpha^2 & 0 & 0 & l_1\alpha^2 \\
	\end{pmatrix}
	$$
\end{lemma}
\begin{proof} It is easily checked that the matrix $\phi$ satisfies the equation $\xi_\tau=\phi\circ^{\tau}\phi^{-1}$.
\end{proof}

%

\begin{theorem}\label{thm-eq2Q}
The set of isomorphism classes of Brauer-Severi surfaces defined over $\mathbb{Q}$ is in bijection with the set of equivalent classes of pairs $(\chi,\alpha)$ where $L$ is a Galois extension of degree $3$, $\chi:\,\operatorname{Gal}(L/\mathbb{Q})\xrightarrow{\sim} \mathbb{Z}/3\mathbb{Z}$ is an isomorphism and $\alpha\in\mathbb{Q}^*$. Two pairs $(\chi,\alpha)$ and $(\chi',\alpha')$ are equivalent if and only if $\chi=\chi'$ and $\alpha'\alpha^{-1}\in\operatorname{Nm}_{L/K}(L^*)$. Given $(\chi,\alpha)$ with $L=\mathbb{Q}(l_1,l_2,l_3)$ where the $l_i$ are conjugate numbers, the corresponding Brauer-Severi variety $\mathcal{B}$ is given by 
the intersection $\cap_{\sigma\in\operatorname{Gal}(L/\Q)}\,^{\sigma}X$, where $X/L$ is the variety in $\P^9$ defined by the set of equations:
\small
$$
\begin{array}{c}
\alpha(l_1\omega_0+l_2\omega_6+l_3\omega_9)(l_2\omega_0+l_3\omega_6+l_1\omega_9)^2=(l_2\omega_1+l_3\omega_5+l_1\omega_7)^3\\
\alpha^3(l_1\omega_1+l_2\omega_5+l_3\omega_7)(l_2\omega_0+l_3\omega_6+l_1\omega_9)^2=(l_2\omega_1+l_3\omega_5+l_1\omega_7)^2(l_3\omega_1+l_1\omega_5+l_2\omega_7)\\
\alpha^2(l_1\omega_2+l_2\omega_3+l_3\omega_8)(l_2\omega_0+l_3\omega_6+l_1\omega_9)^2=(l_2\omega_1+l_3\omega_5+l_1\omega_7)^2(l_2\omega_0+l_3\omega_6+l_1\omega_9)\\
\alpha(l_3\omega_2+l_1\omega_3+l_2\omega_8)(l_2\omega_0+l_3\omega_6+l_1\omega_9)^2=(l_2\omega_1+l_3\omega_5+l_1\omega_7)(l_2\omega_2+l_3\omega_3+l_1\omega_8)^2\\
\omega_4(l_2\omega_0+l_3\omega_6+l_1\omega_9)^2=(l_2\omega_1+l_3\omega_5+l_1\omega_7)(l_2\omega_2+l_3\omega_3+l_1\omega_8)(l_2\omega_0+l_3\omega_6+l_1\omega_9)\\
\alpha^2(l_3\omega_0+l_1\omega_6+l_2\omega_9)(l_2\omega_0+l_3\omega_6+l_1\omega_9)^2=(l_2\omega_2+l_3\omega_3+l_1\omega_8)^3\\
\alpha(l_3\omega_1+l_1\omega_5+l_2\omega_7)(l_2\omega_0+l_3\omega_6+l_1\omega_9)^2=(l_2\omega_2+l_3\omega_3+l_1\omega_8)^2(l_2\omega_0+l_3\omega_6+l_1\omega_9).
\end{array}
$$
\normalsize
\end{theorem}

\begin{proof} By Theorem \ref{thm-cocycle}~we know that all the Brauer-Severi surfaces are parametrized by $\operatorname{Az}_3^{\Q}$. We also know that all the elements in $\operatorname{Az}_3^{\Q}$ are cyclic algebras \cite{Wed}. We use the description in Proposition \ref{prop-cyclic} for cyclic algebras. 
	
	In order to compute the equations, we plug the equation of the isomorphism $\phi$ in Lemma \ref{lemma_iso} into the equations of $V_2(\mathbb{P}^2)\subseteq\mathbb{P}^9$ given in Corollary \ref{cor_eqP2}. These equations are, a priori, not defined over $\mathbb{Q}$, even if the ideal generating $\mathcal{B}$ is. Notice that after plugging $\phi$ in the equations in the second and the third columns in Corollary \ref{cor_eqP2}, we get the conjugate equations to the ones appearing in the first column and shown here. Hence, the intersection $\cap_{\sigma\in\operatorname{Gal}(L/\Q)}\,^{\sigma}X$ gives the equations for $\mathcal{B}/\Q$.
\end{proof}

\subsection{Degree $3$ cyclic extensions of $\Q$}

In order to show equations defined over $\Q$ for the Brauer-Severi surfaces shown in Theorem \ref{thm-eq2Q}, we need to work with a good basis for the degree $3$ cyclic extensions $L$ of $\Q$. Ideally, we would like to find a basis $l_1, l_2,l_3$ of $L/\Q$ such that we could easily write each equation in Theorem \ref{thm-eq2Q} as $l_i^2f_1+l_if_2+f_3=0$ with $f_i\in\Q[\omega_0,...,\omega_9]$. Then, since the matrix
\small{
$$
\begin{pmatrix}
l_1 & l_2 & l_3\\ l_2 & l_3 & l_1 \\ l_3 & l_1 & l_2
\end{pmatrix}
$$}
is invertible, the equations of the Brauer-Severi surface would be given by $f_1=f_2=f_3=0$. 

In this subsection we will find an element $t_1$ in $L$ such that $l_1,\,l_2,\,l_3$ with $l_1=t_1$ is a basis in which we can easily write the equations in Theorem \ref{thm-eq2Q} as $l_i^2f_1+l_if_2+f_3=0$ for some $f_i\in\Q[\omega_0,...,\omega_9]$.

\begin{proposition}\label{prop-ext3} Let $L/\mathbb{Q}$ be a cyclic degree $3$ extension given by the decomposition field of the polynomial $P(t)=t^3+At^2+Bt+C$ with $A,B,C\in\mathbb{Z}$. Let $t_1$ be a fixed root of $P(t)=0$. Then there exist $\alpha,\beta,\gamma,\delta\in\mathbb{Z}$ such that $t_2=\frac{\alpha t_1+\beta}{\gamma t_1+\delta}$ and $t_3=\frac{\alpha t_2+\beta}{\gamma t_2+\delta}$ are the other two roots of $P(t)=0$.
\end{proposition}

\begin{proof} Since $L/\mathbb{Q}$ is Galois and of degree $3$, there exist $a,b,c\in\mathbb{Q}$ such that $t_2=at_1^2+bt_1+c$. We want to look for $\alpha,\beta,\gamma,\delta\in\mathbb{Z}$ such that $(at_1^2+bt_1+c)(\alpha t_1+\beta)=(\gamma t_1+\delta)$. We can take $\alpha$ equal to the product of the numerator and the denominator of $a$ and we make $\beta=(\alpha a A-b\alpha)/a$, $\gamma = \beta b+ c\alpha -\alpha a B$ and $\delta = c\beta-\alpha a C$. Now, it is easy to check that $t_3=\frac{\alpha t_2+\beta}{\gamma t_2+\delta}$ is the third root and that $t_1\neq t_2\neq t_3\neq t_1$.
\end{proof}

\begin{lemma} With the notation above, $M:=\begin{pmatrix}\alpha & \beta \\ \gamma & \delta
	\end{pmatrix} \in \mathcal{M}_3(\mathbb{Z})$ has order $3$,  $\delta=-(\alpha+1)$ and $\beta\gamma=-(\alpha^2+\alpha+1)$. Moreover, if we have $A=0$, then $B=3\frac{\beta}{\gamma}$ and $C=\frac{\beta(2\alpha+1)}{\gamma^2}$.
\end{lemma}

\begin{proof} It is enough with checking that $t_1=\frac{\alpha t_3+\beta}{\gamma t_3+\delta}$, so $M^3=1$, which implies $\alpha\delta-\beta\gamma=1$, $\alpha+\delta=-1$, $\alpha^2+\beta\gamma=\delta$ and $\delta^2+\beta\gamma=\alpha$. For the last statement, we just write $0=-A=t_1+t_2+t_3$.
\end{proof}

\subsection{A particular example}\label{subsec-example}
Let us take $A=0$, $B=-3$ and $C=1$, that is, $\alpha=\gamma=-1$, $\beta=1$ and $\delta=0$, and the cocycle given by
\small{
$$
\sigma\rightarrow\begin{pmatrix}
0 & 1 & 0\\
0 & 0 & 1\\
2 & 0 & 0
\end{pmatrix},
$$}
where $\sigma(t_i)=t_{i+1}$. Notice that $2$ is not a norm in the decomposition field $\mathbb{Q}_P(t)$ since $2$ is inert in $\mathbb{Q}_P(t)$. We prove it by checking that $2$ does not divide the discriminant $\Delta_P=81$ and that the polynomial $P(t)=t^3-3t+1$ is irreducible in $\mathbb{F}_2$. 

The equations in Theorem \ref{thm-eq2Q} for this particular example look like:
\footnotesize
\begin{gather*}
(-2\omega_{0}^3 + 6\omega_{0}^2\omega_{6} - 6\omega_{0}\omega_{6}^2 - 3\omega_{1}^3 + 6\omega_{1}^2\omega_{5} +
    3\omega_{1}^2\omega_{7} - 6\omega_{1}\omega_{5}^2 - 3\omega_{1}\omega_{7}^2 + 3\omega_{5}^3 - 3\omega_{5}^2\omega_{7} +
    3\omega_{5}\omega_{7}^2 +  \\
4\omega_{6}^3 - 6\omega_{6}^2\omega_{9} + 6\omega_{6}\omega_{9}^2 - 2\omega_{9}^3)\pmb{t^2} + (
    2\omega_{0}^3 - 6\omega_{0}^2\omega_{9} + 6\omega_{0}\omega_{9}^2 + 3\omega_{1}^2\omega_{5} - 3\omega_{1}^2\omega_{7} -
    3\omega_{1}\omega_{5}^2 + 3\omega_{1}\omega_{7}^2 + \\3\omega_{5}^3 - 
6\omega_{5}^2\omega_{7} + 6\omega_{5}\omega_{7}^2 + 2\omega_{6}^3
    - 6\omega_{6}^2\omega_{9} + 6\omega_{6}\omega_{9}^2 - 3\omega_{7}^3 - 4\omega_{9}^3)\boldsymbol{t} + 2\omega_{0}^3 - 18\omega_{0}^2\omega_{6} +
    12\omega_{0}^2\omega_{9} + \\24\omega_{0}\omega_{6}^2 - 12\omega_{0}\omega_{6}\omega_{9} -
 6\omega_{0}\omega_{9}^2 + 7\omega_{1}^3 - 18\omega_{1}^2\omega_{5} -
    3\omega_{1}^2\omega_{7} + 15\omega_{1}\omega_{5}^2 + 6\omega_{1}\omega_{5}\omega_{7} - 5\omega_{5}^3 - 3\omega_{5}\omega_{7}^2 -\\ 10\omega_{6}^3 +
    6\omega_{6}^2\omega_{9} + \omega_{7}^3 + 2\omega_{9}^3=0
\end{gather*}
\begin{gather*}
(-2\omega_{0}^2\omega_{1} + 6\omega_{0}^2\omega_{5} - 4\omega_{0}^2\omega_{7} + 4\omega_{0}\omega_{1}\omega_{9} - 8\omega_{0}\omega_{5}\omega_{6}
    - 4\omega_{0}\omega_{5}\omega_{9} + 8\omega_{0}\omega_{6}\omega_{7} - 3\omega_{1}^2\omega_{2} + 2\omega_{1}^2\omega_{3} +\\
    \omega_{1}^2\omega_{8} + 
4\omega_{1}\omega_{2}\omega_{5} + 2\omega_{1}\omega_{2}\omega_{7} - 4\omega_{1}\omega_{3}\omega_{5} +
    2\omega_{1}\omega_{6}^2 - 4\omega_{1}\omega_{6}\omega_{9} - 2\omega_{1}\omega_{7}\omega_{8} - 2\omega_{2}\omega_{5}^2 -
    \omega_{2}\omega_{7}^2 +\\ 3\omega_{3}\omega_{5}^2 - 2\omega_{3}\omega_{5}\omega_{7} + 
\omega_{3}\omega_{7}^2 - \omega_{5}^2\omega_{8} +
    4\omega_{5}\omega_{6}^2 + 2\omega_{5}\omega_{7}\omega_{8} + 2\omega_{5}\omega_{9}^2 - 6\omega_{6}^2\omega_{7} +
    4\omega_{6}\omega_{7}\omega_{9} - 2\omega_{7}\omega_{9}^2)\pmb{t^2} +\\ (2\omega_{0}^2\omega_{1} - 2\omega_{0}^2\omega_{7} - 4\omega_{0}\omega_{1}\omega_{9} -
    4\omega_{0}\omega_{5}\omega_{6} + 4\omega_{0}\omega_{5}\omega_{9} + 4\omega_{0}\omega_{6}\omega_{7} + \omega_{1}^2\omega_{3} - \omega_{1}^2\omega_{8} +
    2\omega_{1}\omega_{2}\omega_{5} - \\2\omega_{1}\omega_{2}\omega_{7} - 2\omega_{1}\omega_{3}\omega_{5} + 4\omega_{1}\omega_{6}^2 - 8\omega_{1}\omega_{6}\omega_{9} +
    2\omega_{1}\omega_{7}\omega_{8} + 6\omega_{1}\omega_{9}^2 - \omega_{2}\omega_{5}^2 + \omega_{2}\omega_{7}^2 + 3\omega_{3}\omega_{5}^2 -\\
    4\omega_{3}\omega_{5}\omega_{7} + 2\omega_{3}\omega_{7}^2 - 2\omega_{5}^2\omega_{8} + 2\omega_{5}\omega_{6}^2 + 4\omega_{5}\omega_{7}\omega_{8} -
    2\omega_{5}\omega_{9}^2 - 
6\omega_{6}^2\omega_{7} +  8\omega_{6}\omega_{7}\omega_{9} - 3\omega_{7}^2\omega_{8} - 4\omega_{7}\omega_{9}^2)\boldsymbol{t} +\\
    2\omega_{0}^2\omega_{1} - 14\omega_{0}^2\omega_{5} + 12\omega_{0}^2\omega_{7} - 4\omega_{0}\omega_{1}\omega_{6} + 24\omega_{0}\omega_{5}\omega_{6} + 4\omega_{0}\omega_{5}\omega_{9}
    -
 20\omega_{0}\omega_{6}\omega_{7} - 4\omega_{0}\omega_{7}\omega_{9} + 7\omega_{1}^2\omega_{2} -\\ 6\omega_{1}^2\omega_{3} - \omega_{1}^2\omega_{8} - 12\omega_{1}\omega_{2}\omega_{5} -
    2\omega_{1}\omega_{2}\omega_{7} + 10\omega_{1}\omega_{3}\omega_{5} + 2\omega_{1}\omega_{3}\omega_{7} + 2\omega_{1}\omega_{5}\omega_{8} + 
4\omega_{1}\omega_{6}\omega_{9} - 2\omega_{1}\omega_{9}^2
    + \\5\omega_{2}\omega_{5}^2 + 2\omega_{2}\omega_{5}\omega_{7} - 5\omega_{3}\omega_{5}^2 - \omega_{3}\omega_{7}^2 - 10\omega_{5}\omega_{6}^2 - 4\omega_{5}\omega_{6}\omega_{9} -
    2\omega_{5}\omega_{7}\omega_{8} + 10\omega_{6}^2\omega_{7} + \omega_{7}^2\omega_{8} +
 2\omega_{7}\omega_{9}^2=0
 \end{gather*}
 \begin{gather*}
(-\omega_{0}^2\omega_{2} + 3\omega_{0}^2\omega_{3} - 2\omega_{0}^2\omega_{8} - 3\omega_{0}\omega_{1}^2 + 4\omega_{0}\omega_{1}\omega_{5} +
    2\omega_{0}\omega_{1}\omega_{7} + 2\omega_{0}\omega_{2}\omega_{9} - 4\omega_{0}\omega_{3}\omega_{6} - 2\omega_{0}\omega_{3}\omega_{9} -\\
    2\omega_{0}\omega_{5}^2 +
 4\omega_{0}\omega_{6}\omega_{8} - \omega_{0}\omega_{7}^2 + 2\omega_{1}^2\omega_{6} + \omega_{1}^2\omega_{9} -
    4\omega_{1}\omega_{5}\omega_{6} - 2\omega_{1}\omega_{7}\omega_{9} + \omega_{2}\omega_{6}^2 - 2\omega_{2}\omega_{6}\omega_{9} +
    2\omega_{3}\omega_{6}^2 +\\ \omega_{3}\omega_{9}^2 + 3\omega_{5}^2\omega_{6} -
 \omega_{5}^2\omega_{9} - 2\omega_{5}\omega_{6}\omega_{7} +
    2\omega_{5}\omega_{7}\omega_{9} - 3\omega_{6}^2\omega_{8} + \omega_{6}\omega_{7}^2 + 2\omega_{6}\omega_{8}\omega_{9} - \omega_{8}\omega_{9}^2)\pmb{t^2}
    + (\omega_{0}^2\omega_{2} -\\ \omega_{0}^2\omega_{8} + 2\omega_{0}\omega_{1}\omega_{5} - 2\omega_{0}\omega_{1}\omega_{7} -
 2\omega_{0}\omega_{2}\omega_{9} -
    2\omega_{0}\omega_{3}\omega_{6} + 2\omega_{0}\omega_{3}\omega_{9} - \omega_{0}\omega_{5}^2 + 2\omega_{0}\omega_{6}\omega_{8} + \omega_{0}\omega_{7}^2 +
    \omega_{1}^2\omega_{6} -\\ \omega_{1}^2\omega_{9} - 2\omega_{1}\omega_{5}\omega_{6} + 2\omega_{1}\omega_{7}\omega_{9} + 2\omega_{2}\omega_{6}^2 -
    4\omega_{2}\omega_{6}\omega_{9} + 3\omega_{2}\omega_{9}^2 + \omega_{3}\omega_{6}^2 - \omega_{3}\omega_{9}^2 + 3\omega_{5}^2\omega_{6} -
    2\omega_{5}^2\omega_{9} -\\ 4\omega_{5}\omega_{6}\omega_{7} + 4\omega_{5}\omega_{7}\omega_{9} - 3\omega_{6}^2\omega_{8} + 2\omega_{6}\omega_{7}^2 +
    4\omega_{6}\omega_{8}\omega_{9} - 
3\omega_{7}^2\omega_{9} - 2\omega_{8}\omega_{9}^2)\pmb{t} + \omega_{0}^2\omega_{2} - 7\omega_{0}^2\omega_{3} + 6\omega_{0}^2\omega_{8} +\\
    7\omega_{0}\omega_{1}^2 - 12\omega_{0}\omega_{1}\omega_{5} - 2\omega_{0}\omega_{1}\omega_{7} - 2\omega_{0}\omega_{2}\omega_{6} + 12\omega_{0}\omega_{3}\omega_{6} + 
2\omega_{0}\omega_{3}\omega_{9}
    + 5\omega_{0}\omega_{5}^2 + 2\omega_{0}\omega_{5}\omega_{7} - 10\omega_{0}\omega_{6}\omega_{8} - \\2\omega_{0}\omega_{8}\omega_{9} - 6\omega_{1}^2\omega_{6} - \omega_{1}^2\omega_{9} +
    10\omega_{1}\omega_{5}\omega_{6} + 2\omega_{1}\omega_{5}\omega_{9} + 2\omega_{1}\omega_{6}\omega_{7} + 
2\omega_{2}\omega_{6}\omega_{9} - \omega_{2}\omega_{9}^2 - 5\omega_{3}\omega_{6}^2 -\\
    2\omega_{3}\omega_{6}\omega_{9} - 5\omega_{5}^2\omega_{6} - 2\omega_{5}\omega_{7}\omega_{9} + 5\omega_{6}^2\omega_{8} - \omega_{6}\omega_{7}^2 + \omega_{7}^2\omega_{9} +
    \omega_{8}\omega_{9}^2=0
    \end{gather*}
    \begin{gather*}
(-4\omega_{0}^2\omega_{2} - 2\omega_{0}^2\omega_{3} + 6\omega_{0}^2\omega_{8} + 8\omega_{0}\omega_{2}\omega_{6} + 4\omega_{0}\omega_{3}\omega_{9}
    - 8\omega_{0}\omega_{6}\omega_{8} - 4\omega_{0}\omega_{8}\omega_{9} - 3\omega_{1}\omega_{2}^2 + 4\omega_{1}\omega_{2}\omega_{3} +\\
    2\omega_{1}\omega_{2}\omega_{8} - 
2\omega_{1}\omega_{3}^2 - \omega_{1}\omega_{8}^2 + 2\omega_{2}^2\omega_{5} + \omega_{2}^2\omega_{7} -
    4\omega_{2}\omega_{3}\omega_{5} - 6\omega_{2}\omega_{6}^2 + 4\omega_{2}\omega_{6}\omega_{9} - 2\omega_{2}\omega_{7}\omega_{8} -
    2\omega_{2}\omega_{9}^2 +\\ 3\omega_{3}^2\omega_{5} - \omega_{3}^2\omega_{7} -
2\omega_{3}\omega_{5}\omega_{8} + 2\omega_{3}\omega_{6}^2
    - 4\omega_{3}\omega_{6}\omega_{9} + 2\omega_{3}\omega_{7}\omega_{8} + \omega_{5}\omega_{8}^2 + 4\omega_{6}^2\omega_{8} +
    2\omega_{8}\omega_{9}^2)\pmb{t^2} - (2\omega_{0}^2\omega_{2} +\\ 2\omega_{0}^2\omega_{3} + 4\omega_{0}\omega_{2}\omega_{6} - 
4\omega_{0}\omega_{3}\omega_{9} -
    4\omega_{0}\omega_{6}\omega_{8} + 4\omega_{0}\omega_{8}\omega_{9} + 2\omega_{1}\omega_{2}\omega_{3} - 2\omega_{1}\omega_{2}\omega_{8} - \omega_{1}\omega_{3}^2 +
    \omega_{1}\omega_{8}^2 + \omega_{2}^2\omega_{5} -\\ \omega_{2}^2\omega_{7} - 2\omega_{2}\omega_{3}\omega_{5} - 6\omega_{2}\omega_{6}^2 +
    8\omega_{2}\omega_{6}\omega_{9} + 2\omega_{2}\omega_{7}\omega_{8} - 4\omega_{2}\omega_{9}^2 + 3\omega_{3}^2\omega_{5} - 2\omega_{3}^2\omega_{7} -
    4\omega_{3}\omega_{5}\omega_{8} + 4\omega_{3}\omega_{6}^2 -\\ 8\omega_{3}\omega_{6}\omega_{9} + 4\omega_{3}\omega_{7}\omega_{8} + 6\omega_{3}\omega_{9}^2 +
    2\omega_{5}\omega_{8}^2 + 2\omega_{6}^2\omega_{8} - 3\omega_{7}\omega_{8}^2 - 2\omega_{8}\omega_{9}^2)\pmb{t} + 12\omega_{0}^2\omega_{2} +
    2\omega_{0}^2\omega_{3} - 14\omega_{0}^2\omega_{8} -\\ 20\omega_{0}\omega_{2}\omega_{6} - 4\omega_{0}\omega_{2}\omega_{9} - 4\omega_{0}\omega_{3}\omega_{6} + 24\omega_{0}\omega_{6}\omega_{8}
    + 4\omega_{0}\omega_{8}\omega_{9} + 7\omega_{1}\omega_{2}^2 - 12\omega_{1}\omega_{2}\omega_{3} - 2\omega_{1}\omega_{2}\omega_{8} + 5\omega_{1}\omega_{3}^2 + \\2\omega_{1}\omega_{3}\omega_{8}
    - 6\omega_{2}^2\omega_{5} - \omega_{2}^2\omega_{7} + 10\omega_{2}\omega_{3}\omega_{5} + 
2\omega_{2}\omega_{3}\omega_{7} + 2\omega_{2}\omega_{5}\omega_{8} + 10\omega_{2}\omega_{6}^2 +
    2\omega_{2}\omega_{9}^2 - 5\omega_{3}^2\omega_{5} +\\ 4\omega_{3}\omega_{6}\omega_{9} - 2\omega_{3}\omega_{7}\omega_{8} - 2\omega_{3}\omega_{9}^2 - \omega_{5}\omega_{8}^2 -
    10\omega_{6}^2\omega_{8} -
 4\omega_{6}\omega_{8}\omega_{9} + \omega_{7}\omega_{8}^2=0
 \end{gather*}
 \begin{gather*}
(-\omega_{0}^2\omega_{4} - 3\omega_{0}\omega_{1}\omega_{2} + 2\omega_{0}\omega_{1}\omega_{3} + \omega_{0}\omega_{1}\omega_{8} + 2\omega_{0}\omega_{2}\omega_{5} +
    \omega_{0}\omega_{2}\omega_{7} - 2\omega_{0}\omega_{3}\omega_{5} + 2\omega_{0}\omega_{4}\omega_{6} - 
\omega_{0}\omega_{7}\omega_{8} +\\
    2\omega_{1}\omega_{2}\omega_{6} + \omega_{1}\omega_{2}\omega_{9} - 2\omega_{1}\omega_{3}\omega_{6} - \omega_{1}\omega_{8}\omega_{9} -
    2\omega_{2}\omega_{5}\omega_{6} - \omega_{2}\omega_{7}\omega_{9} + 3\omega_{3}\omega_{5}\omega_{6} - \omega_{3}\omega_{5}\omega_{9} - \omega_{3}\omega_{6}\omega_{7} + \\
\omega_{3}\omega_{7}\omega_{9} - 2\omega_{4}\omega_{6}\omega_{9} + \omega_{4}\omega_{9}^2 - \omega_{5}\omega_{6}\omega_{8} + \omega_{5}\omega_{8}\omega_{9}
    + \omega_{6}\omega_{7}\omega_{8})\pmb{t^2} - (\omega_{0}^2\omega_{4} + \omega_{0}\omega_{1}\omega_{3} - \omega_{0}\omega_{1}\omega_{8} +\\ \omega_{0}\omega_{2}\omega_{5} - 
    \omega_{0}\omega_{2}\omega_{7} - \omega_{0}\omega_{3}\omega_{5} + 2\omega_{0}\omega_{4}\omega_{9} + \omega_{0}\omega_{7}\omega_{8} + \omega_{1}\omega_{2}\omega_{6} -
    \omega_{1}\omega_{2}\omega_{9} - \omega_{1}\omega_{3}\omega_{6} + \omega_{1}\omega_{8}\omega_{9} -\\ \omega_{2}\omega_{5}\omega_{6} + \omega_{2}\omega_{7}\omega_{9} +
    3\omega_{3}\omega_{5}\omega_{6} - 2\omega_{3}\omega_{5}\omega_{9} - 2\omega_{3}\omega_{6}\omega_{7} + 2\omega_{3}\omega_{7}\omega_{9} + \omega_{4}\omega_{6}^2 -
    2\omega_{4}\omega_{6}\omega_{9} - 2\omega_{5}\omega_{6}\omega_{8} +\\ 2\omega_{5}\omega_{8}\omega_{9} + 2\omega_{6}\omega_{7}\omega_{8} - 
3\omega_{7}\omega_{8}\omega_{9})\pmb{t} +
    4\omega_{0}^2\omega_{4} + 7\omega_{0}\omega_{1}\omega_{2} - 6\omega_{0}\omega_{1}\omega_{3} - \omega_{0}\omega_{1}\omega_{8} - 6\omega_{0}\omega_{2}\omega_{5} - \\ \omega_{0}\omega_{2}\omega_{7} +
    5\omega_{0}\omega_{3}\omega_{5} + \omega_{0}\omega_{3}\omega_{7} -
 6\omega_{0}\omega_{4}\omega_{6} - 2\omega_{0}\omega_{4}\omega_{9} + \omega_{0}\omega_{5}\omega_{8} - 6\omega_{1}\omega_{2}\omega_{6} -
    \omega_{1}\omega_{2}\omega_{9} + 5\omega_{1}\omega_{3}\omega_{6} +\\ \omega_{1}\omega_{3}\omega_{9} + \omega_{1}\omega_{6}\omega_{8} + 5\omega_{2}\omega_{5}\omega_{6} + \omega_{2}\omega_{5}\omega_{9} +
    \omega_{2}\omega_{6}\omega_{7} - 5\omega_{3}\omega_{5}\omega_{6} - \omega_{3}\omega_{7}\omega_{9} + 2\omega_{4}\omega_{6}^2 + 2\omega_{4}\omega_{6}\omega_{9} -\\ \omega_{5}\omega_{8}\omega_{9} -
    \omega_{6}\omega_{7}\omega_{8} + \omega_{7}\omega_{8}\omega_{9}=0
    \end{gather*}
    \begin{gather*}
(-8\omega_{0}^3 + 12\omega_{0}^2\omega_{6} + 
12\omega_{0}^2\omega_{9} - 12\omega_{0}\omega_{6}^2 - 12\omega_{0}\omega_{9}^2
    - 3\omega_{2}^3 + 6\omega_{2}^2\omega_{3} + 3\omega_{2}^2\omega_{8} - 6\omega_{2}\omega_{3}^2 - 3\omega_{2}\omega_{8}^2
    + 3\omega_{3}^3 -\\ 3\omega_{3}^2\omega_{8} + 
3\omega_{3}\omega_{8}^2 + 4\omega_{6}^3 + 4\omega_{9}^3)\pmb{t^2} + (-
    4\omega_{0}^3 + 12\omega_{0}^2\omega_{6} - 12\omega_{0}\omega_{6}^2 + 3\omega_{2}^2\omega_{3} - 3\omega_{2}^2\omega_{8} -
    3\omega_{2}\omega_{3}^2 + 3\omega_{2}\omega_{8}^2 +\\ 3\omega_{3}^3 - 6\omega_{3}^2\omega_{8} + 
6\omega_{3}\omega_{8}^2 + 8\omega_{6}^3
    - 12\omega_{6}^2\omega_{9} + 12\omega_{6}\omega_{9}^2 - 3\omega_{8}^3 - 4\omega_{9}^3)\pmb{t} + 24\omega_{0}^3 - 36\omega_{0}^2\omega_{6} -
    36\omega_{0}^2\omega_{9} +\\ 12\omega_{0}\omega_{6}^2 + 48\omega_{0}\omega_{6}\omega_{9} + 
12\omega_{0}\omega_{9}^2 + 7\omega_{2}^3 - 18\omega_{2}^2\omega_{3} -
    3\omega_{2}^2\omega_{8} + 15\omega_{2}\omega_{3}^2 + 6\omega_{2}\omega_{3}\omega_{8} - 5\omega_{3}^3 - 3\omega_{3}\omega_{8}^2 -\\ 12\omega_{6}^2\omega_{9} -
    12\omega_{6}\omega_{9}^2 + \omega_{8}^3=0
    \end{gather*}
    \begin{gather*}
(-4\omega_{0}^2\omega_{1} - 2\omega_{0}^2\omega_{5} + 6\omega_{0}^2\omega_{7} + 8\omega_{0}\omega_{1}\omega_{6} - 3\omega_{0}\omega_{2}^2
    + 4\omega_{0}\omega_{2}\omega_{3} + 2\omega_{0}\omega_{2}\omega_{8} - 2\omega_{0}\omega_{3}^2 \\+ 4\omega_{0}\omega_{5}\omega_{9} -
    8\omega_{0}\omega_{6}\omega_{7} - 
4\omega_{0}\omega_{7}\omega_{9} - \omega_{0}\omega_{8}^2 - 6\omega_{1}\omega_{6}^2 +
    4\omega_{1}\omega_{6}\omega_{9} - 2\omega_{1}\omega_{9}^2 + 2\omega_{2}^2\omega_{6} + \omega_{2}^2\omega_{9} -\\
    4\omega_{2}\omega_{3}\omega_{6} - 2\omega_{2}\omega_{8}\omega_{9} + 3\omega_{3}^2\omega_{6} - \omega_{3}^2\omega_{9} -
    2\omega_{3}\omega_{6}\omega_{8} + 2\omega_{3}\omega_{8}\omega_{9} + 2\omega_{5}\omega_{6}^2 - 4\omega_{5}\omega_{6}\omega_{9} +
    4\omega_{6}^2\omega_{7} +\\ \omega_{6}\omega_{8}^2 + 2\omega_{7}\omega_{9}^2)\pmb{t^2}+( - 2\omega_{0}^2\omega_{1} + 2\omega_{0}^2\omega_{5} +
    4\omega_{0}\omega_{1}\omega_{6} + 
2\omega_{0}\omega_{2}\omega_{3} - 2\omega_{0}\omega_{2}\omega_{8} - \omega_{0}\omega_{3}^2 - 4\omega_{0}\omega_{5}\omega_{9} -\\
    4\omega_{0}\omega_{6}\omega_{7} + 4\omega_{0}\omega_{7}\omega_{9} + \omega_{0}\omega_{8}^2 - 6\omega_{1}\omega_{6}^2 + 8\omega_{1}\omega_{6}\omega_{9} -
    4\omega_{1}\omega_{9}^2 +
 \omega_{2}^2\omega_{6} - \omega_{2}^2\omega_{9} - 2\omega_{2}\omega_{3}\omega_{6} + 2\omega_{2}\omega_{8}\omega_{9} +\\
    3\omega_{3}^2\omega_{6} - 2\omega_{3}^2\omega_{9} - 4\omega_{3}\omega_{6}\omega_{8} + 4\omega_{3}\omega_{8}\omega_{9} + 4\omega_{5}\omega_{6}^2 -
    8\omega_{5}\omega_{6}\omega_{9} + 6\omega_{5}\omega_{9}^2 +
 2\omega_{6}^2\omega_{7} + 2\omega_{6}\omega_{8}^2 - 2\omega_{7}\omega_{9}^2 -\\
    3\omega_{8}^2\omega_{9} )\pmb{t}+ 12\omega_{0}^2\omega_{1} + 2\omega_{0}^2\omega_{5} - 14\omega_{0}^2\omega_{7} - 20\omega_{0}\omega_{1}\omega_{6} - 4\omega_{0}\omega_{1}\omega_{9}
    + 7\omega_{0}\omega_{2}^2 - 
12\omega_{0}\omega_{2}\omega_{3} - 2\omega_{0}\omega_{2}\omega_{8} +\\ 5\omega_{0}\omega_{3}^2 + 2\omega_{0}\omega_{3}\omega_{8} - 4\omega_{0}\omega_{5}\omega_{6}
    +  24\omega_{0}\omega_{6}\omega_{7} + 4\omega_{0}\omega_{7}\omega_{9} + 10\omega_{1}\omega_{6}^2 + 2\omega_{1}\omega_{9}^2 -
 6\omega_{2}^2\omega_{6} - \omega_{2}^2\omega_{9} +\\
    10\omega_{2}\omega_{3}\omega_{6} + 2\omega_{2}\omega_{3}\omega_{9} + 2\omega_{2}\omega_{6}\omega_{8} - 5\omega_{3}^2\omega_{6} - 2\omega_{3}\omega_{8}\omega_{9} + 4\omega_{5}\omega_{6}\omega_{9}
    - 2\omega_{5}\omega_{9}^2 - 10\omega_{6}^2\omega_{7} - \\
4\omega_{6}\omega_{7}\omega_{9} - \omega_{6}\omega_{8}^2 + \omega_{8}^2\omega_{9}=0.
\end{gather*}

\normalsize

Now the intersection $\cap_{\sigma\in\operatorname{Gal}(L/\Q)}\,^{\sigma}X$ is just generated by the equations $f_1=f_2=f_3=0$ where the previous equations are written as $f_1\pmb{t^2}+f_2\pmb{t}+f_3=0$ with $f_i\in\Q[\omega_0,\omega_1,\omega_2,\omega_3,\omega_4,\omega_5,\omega_6,\omega_7,\omega_8,\omega_9]$.

\section{Singular equations}
As we have already seen, equations for smooth models of Brauer-Severi varieties are quite impractical since they have many variables and terms. However, if we work with singular models we can show "nicer" models, meaning, having less variables and shorter equations.

\begin{lemma}\label{singular_modelPn} Let $K$ be a perfect field, and let be $B_n:\,X_{1}X_{2}...X_{n+1}=X_{0}^{n+1}\subseteq\mathbb{P}^{n+1}_K$. Then $B_n$ is birationally equivalent to $\mathbb{P}^{n}$ over $\mathbb{Q}$.
\end{lemma}
\begin{proof} Let us consider de map $\psi:\,\mathbb{P}^n\rightarrow\mathbb{P}^{n+1}:\,(x_0:x_1:...:x_n)\mapsto (x_0:x_1:...:x_n:\frac{x_{0}^{n+1}}{x_1...x_{n}})$. It gives a birational map between $\mathbb{P}^n$ and $\psi(\mathbb{P}^n)$. Moreover, $\psi(\mathbb{P}^n)\simeq B_n$. 
\end{proof}

\begin{theorem}\label{thm-singularmodel}
	Let $(\chi,\alpha)$ be a pair consisting on an isomorphism $\chi:\,\operatorname{Gal}(L/K)\xrightarrow{\sim} \mathbb{Z}/(n+1)\mathbb{Z}$, where $L$ is a cyclic Galois extension $L/K$ of degree $n+1$, and an element $\alpha\in K^{*}$. Let $\{l_1,l_2,...,l_{n+1}\}$ be a normal basis of $L/K$. Then, the Brauer-Severi variety associated to $(\chi,\alpha)$ as in Theorem \ref{thm-cocycle} and Proposition \ref{prop-cyclic} is birationally equivalent over $K$ to
	$$
	\operatorname{N}_{L/k}(l_1x_1+...+l_{n+1}x_{n+1})=\alpha x_{0}^{n+1}.
	$$
\end{theorem}

\begin{proof} First of all, notice that this variety is birationally equivalent to $B_n$ in Lemma \ref{singular_modelPn} over $L$. In particular, it is birationally equivalent to $\mathbb{P}^n$ over $L$, and hence birationally equivalent to a Brauer-Severi variety over $K$. We will see that indeed, it is birationally equivalent over $K$ to the Brauer-Severi variety associated to $(\chi,\alpha)$ as in Theorem \ref{thm-cocycle} and Proposition \ref{prop-cyclic}. We will see that a  birational map from $\alpha\operatorname{N}_{L/k}(l_1x_1+...+l_{n+1}x_{n+1})=x_{0}^{n+1}$ to $B_n$ is given by the matrix 
$$
\tilde{\phi}=\begin{pmatrix}\alpha & 0 & 0 & 0 & 0\\ 0 & l_1 & l_2 & ... & l_n  \\ 0 &\alpha l_n & \alpha l_1 &... & \alpha l_{n-1} \\0 & \alpha l_{n-1} & \alpha l_n &... & \alpha l_{n-2} \\ ... & ... & ... & ... & ...\\ 0 &\alpha l_2 & \alpha l_{3}& ... & \alpha l_1 
 \end{pmatrix}.
$$ 
Let us call $\phi_n$ to the composition of $\tilde{\phi}$ with the inverse of the birational map $\psi$ in Lemma \ref{singular_modelPn} between $B_n$ and $\mathbb{P}^n$. Then the cocycle defined by $\tilde{\xi}(\sigma)=\phi_n\,^{\sigma}\phi_{n}^{-1}$ is not equal, but equivalent, to the cocycle in proposition \ref{prop-cyclic} defining the Brauer-Severi variety attached to $(\chi,\alpha)$. Indeed, take 
$$
f(x_0:x_1:...:x_n)=(P:P_0:P_1:...:P_{n-1})\in\operatorname{Aut}_{birat}(\mathbb{P}^n),
$$
where $P=x_0\cdots x_n$ and $P_i=P\frac{x_{i-1}}{x_i}$. It is straightforward to check that $f\xi(\sigma)\,^{\sigma} f^{-1}=\tilde{\xi}(\sigma)$. 
\end{proof}

\begin{corollary} Let $\mathcal{B}$ be a Brauer-Severi surface defined over $\mathbb{Q}$. It corresponds to a pair $(\chi, \alpha)$ where $\chi:\,\operatorname{Gal}(L/\Q)\xrightarrow{\sim} \Z/3\Z$ is an isomorphism, $L$ is a cyclic Galois extension of degree $3$ and  $\alpha\in\mathbb{Q}^*$. Write $L=\mathbb{Q}(l_1,l_2,l_3)$ with $l_i$ a normal basis for $L/\Q$ with minimal polynomial $x^3+Ax^2+Bx+C$. Then $\mathcal{B}$ is given by the singular model 
	$$
	\operatorname{N}_{L/\Q}(l_1x_1+l_2x_2+l_3x_3)=\alpha x_0^3\subseteq\mathbb{P}^3,
	$$
or equivalently by
\small
$$
- C(x_1^3+x_2^3+x_3^3)+ D_1(x_1^2x_2+x_2^2x_3+x_3^2x_1)+ D_2(x_1x_2^2+x_2x_3^2+x_3x_1^2)+$$
$$+(3AB-A^3)x_1x_2x_3=\alpha x_{0}^3\subseteq\mathbb{P}^3,
$$
\normalsize
where $D_1$ is the rational number $l_1^2l_2+l_2^2l_3+l_3^2l_1$ and $D_2=l_1l_2^2+l_2l_3^2+l_3l_1^2$.\footnote{The discriminant of the polynomial $x^3+Ax^2+Bx+C=0$ is $(D_1-D_2)^2$ and $D_1+D_2$ is a symmetric function on $l_1,l_2,l_3$, hence $D_1$ and $D_2$ are writtenable in terms of $A,B$ and $C$.}
\end{corollary}

\begin{example}The singular model for the Brauer-Severi surface in subsection \ref{subsec-example} is
	\small
	$$
	(x_1^3+x_2^3+x_3^3)-6(x_1^2x_2+x_2^2x_3+x_3^2x_1)+3(x_1x_2^2+x_2x_3^2+x_3x_1^2)=2x_{0}^3\subseteq\mathbb{P}^3.
	$$
	\normalsize
\end{example}

\end{document}